\documentclass[11pt,oneside]{amsart}

\usepackage[letterpaper,width=480pt,top=84pt,bottom=60pt,bindingoffset=0pt]{geometry}

\usepackage{hyperref}

\numberwithin{equation}{section}
\usepackage{latexsym,amssymb,amsmath,amsthm,amsfonts,amscd}
\usepackage{enumerate}
\usepackage{graphicx}
\usepackage{tikz}
\usepackage{float}

\usetikzlibrary{matrix,arrows}
\usepackage{multicol}
\usepackage{multirow}
\usepackage{tabularx}

\usepackage{listings}
\definecolor{VerdeOlivo}{rgb}{0.3,0.5,0.1}
\definecolor{Magenta}{rgb}{.65,0.15,.2}
\definecolor{Gris}{gray}{0.3}

\newtheorem{Theorem}{Theorem}[section] 
\newtheorem{Definition}[Theorem]{Definition}
\newtheorem{Proposition}[Theorem]{Proposition}  
\newtheorem{Lemma}[Theorem]{Lemma} 
\newtheorem{Corollary}[Theorem]{Corollary}
\newtheorem{Remark}[Theorem]{Remark}
\newtheorem{Example}[Theorem]{Example}
    
\newtheorem{Conjecture}[Theorem]{Conjecture} 
\newtheorem{Algorithm}[Theorem]{Algorithm}

\theoremstyle{definition}
\newcommand{\bff}[1]{{\bf #1}}

\begin{document}
	
	
\title[Algorithmic aspects of arithmetical structures]{Algorithmic aspects of arithmetical structures}
	
	
\author{Carlos E. Valencia}
\email[C. E. ~Valencia]{cvalencia@math.cinvestav.edu.mx, cvalencia75@gmail.com}
\author{Ralihe R. Villagr\'an}
\email[R. R. ~Villagr\'{a}n]{rvillagran@math.cinvestav.mx, ralihemath@gmail.com }
\thanks{Carlos E. Valencia was partially supported by
SNI and Ralihe R. Villagr\'an by CONACyT}
\address{
Departamento de Matem\'aticas\\
Centro de Investigaci\'on y de Estudios Avanzados del IPN\\
Apartado Postal 14--740 \\
07000 Mexico City, D.F. 
} 
	
\maketitle
\vspace{-8mm}	
\begin{abstract} 
Arithmetical structures on graphs were first introduced in~\cite{Lorenzini89}.
Later in~\cite{arithmetical} they were further studied in the setting of square non-negative integer matrices. 
In both cases, necessary and sufficient conditions for the finiteness of the set of arithmetical structures were given.

More precisely, an arithmetical structure on a non-negative integer matrix $L$ with zero diagonal is a pair $(\mathbf{d},\mathbf{r})\in \mathbb{N}_+^n\times \mathbb{N}_+^n$ such that
\[
(\textrm{Diag}(\mathbf{d})-L)\mathbf{r}^t=\mathbf{0}^t\text{  and }\gcd(r_1,\ldots,r_n)=1.
\]
Thus, arithmetical structures on $L$ are solutions of the polynomial Diophantine equation 
\[
f_L(X):=\det(\text{Diag}(X)-L)=0.
\]
Therefore, it is of interest to ask for an algorithm that compute them.

We present an algorithm that computes arithmetical structures on a square integer non-negative matrix $L$ with zero diagonal.
In order to do this we introduce a new class of Z-matrices, which we call quasi $M$-matrices.
\end{abstract}



	
\section{Introduction}\label{intro}
Given a non-negative integer matrix $L$ with zero diagonal, a pair $(\mathbf{d},\mathbf{r})\in \mathbb{N}_+^n\times \mathbb{N}_+^n$ is called an arithmetical structure of $L$ if 
\[
(\textrm{Diag}(\mathbf{d})-L)\mathbf{r}^t=\mathbf{0}^t\text{  and }\gcd(r_1,\ldots,r_n)=1.
\]
We impose the condition of primitiveness, $\gcd(r_1,\ldots,r_n)=1$, on the vector $\bff{r}$ because $(\textrm{Diag}(\mathbf{d})-L)\mathbf{r}^t=\mathbf{0}^t$ implies that $(\textrm{Diag}(\mathbf{d})-L)c\mathbf{r}^t=\mathbf{0}^t$ for all $c\in \mathbb{N}_+$ and therefore any common factor in the entries of $\bff{r}$ is irredundant.
 
Arithmetical structures were first introduced for graphs (more precisely when $L$ is the adjacency matrix of a graph) by D. Lorenzini in \cite{Lorenzini89} as some intersection matrices that arise in the study of degenerating curves in algebraic geometry. 
A more combinatorial aspect of arithmetical structures on graphs have been studied in~\cite{PathsCycles} and~\cite{connectivityone}.
Unless otherwise specified, $L$ will always denote a square integer non-negative matrix of size $n$ with zero diagonal.

It is important to recall that the set of arithmetical structures on irreducible integer non-negative matrices is finite. 

\begin{Theorem}{\cite[Theorem 3.8]{arithmetical}}\label{finiteness}
If $L$ is an integer non-negative square matrix with zero diagonal, then the set of arithmetical structures $\mathcal{A}(L)$ on $L$ is finite if and only if $L$ is irreducible. 
\end{Theorem}

A matrix $A$ is called reducible whenever there exists a permutation matrix $P$ such that:
\[
P^tAP=\begin{pmatrix}
X & Y\\
0 & Z
\end{pmatrix}.
\]
That is, $A$ is similar via a permutation to a block upper triangular matrix. 
We say that $A$ is irreducible when it is not reducible. 
Equivalently, when $A$ is the adjacency matrix of a digraph, $A$ is irreducible if and only if the digraph associated to $A$ is strongly connected. 
\begin{Remark}
When $L$ is a block matrix its arithmetical structures can be obtained from the arithmetical structures of its diagonal blocks. 
Something similar when $L$ is reducible.
\end{Remark}

Since the set of arithmetical structures is finite, it is natural to ask if there exists an algorithm that computes them.
We recall that every vector $\bff{d}$ of an arithmetical structure $(\bff{d}, \bff{r})$ of the matrix $L$ is a solution of the polynomial Diophantine equation
\[
f_L(X):=\det(\textrm{Diag}(\mathbf{X})-L)=0.
\]

However, its important to note that not any solution of this Diophantine equation is an arithmetical structure on $L$, see~\cite[Remark 3.11]{arithmetical}.
Therefore computing arithmetical structures of a matrix consist on computing a subset of the solutions of a very special class of Diophantine equations, those whose polynomial is the determinant of a matrix with variables in the diagonal.

The tenth problem on Hilbert's list asks to find an algorithm such that, given any polynomial Diophantine equation, it determines whether it has a solution over the integers. 
Based on the important preliminary work by Martin Davis, Hilary Putnam and Julia Robinson, Yuri V. Matiyasevich showed in 1970 that tenth problem on Hilbert's has a negative answer. 
That is, that in general no such algorithm exists, see~\cite{Matiyasevich}. 

The purpose of this article is to give an algorithm that computes the arithmetical structures of an integer non-negative square matrix with zero diagonal.
This article is divided in two sections.
In section~\ref{ae}, we recall some theory about $M$-matrices as given in~\cite{arithmetical}. 
Additionally we introduce the class of quasi $M$-matrices. 
These matrices have all their proper principal minors being positive, but unlike $M$-matrices and almost non-singular $M$-matrix its determinant is not necessarily non-negative.
Moreover, we will establish some properties of these matrices that help us to find the algorithm and prove its correctness. 

Now, let
\[
\mathcal{D}_{\geq 0}(L)=\{ \bff{d}\in\mathbb{N}_+^n\ |\ (\text{Diag}(\mathbf{d})-L)\text{ is an almost non-singular }M\text{-matrix} \},
\]
where $L$ is a square integer non-negative matrix $L$ with zero diagonal.
By Dickson's Lemma the set of minimal elements $\min\mathcal{D}_{\geq 0}(L)$ of $\mathcal{D}_{\geq 0}(L)$ is finite. 
In Section~\ref{algo}, we present an algorithm that computes $\min\mathcal{D}_{\geq 0}(L)$, see Algorithm~\ref{A1}. 
Using this algorithm as subroutine we get a second algorithm that computes the arithmetical structures on $L$, see Algorithm~\ref{AA}.

At the end of this article we use the algorithm developed to present some computational evidence for the following conjecture.
This data will give an idea of the practical complexity of the problem of compute arithmetical structures.
\begin{Conjecture}\cite[Conjecture 6.10]{arithmetical}
Let $G$ be a simple graph with $n$ vertices, then 
\[
\Big|\,\mathcal{A}(P_n)\,\Big| \leq \Big|\,\mathcal{A}(G)\,\Big| \leq \Big|\,\mathcal{A}(K_n)\,\Big|,
\]
where $P_n$ and $K_n$ are the path and the complete graph on $n$ vertices respectively.
\end{Conjecture}

Throughout this article we use the usual partial order over $\mathbb{R}^n$ given by $\mathbf{a}\leq\mathbf{b}$ if and only if $\bff{a}_i\leq \bff{b}_i$ for all $1\leq i \leq n$ and $\mathbf{a}$,$\mathbf{b}\in\mathbb{R}^n$. 
In a similar way, $\mathbf{a}<\mathbf{b}$ if and only if $\mathbf{a}\leq\mathbf{b}$ and $\mathbf{a}\neq\mathbf{b}$.
It is well known that this is a well partial order over $\mathbb{N}^n$. 
The following property of subsets of $\mathbb{N}^n$ under this partial order is known as Dickson's Lemma.

\begin{Lemma}\cite{Dickson}
For any $S\subseteq \mathbb{N}^n$, the set  
\[
\min(S)=\{ \mathbf{x}\in S \, | \, \mathbf{y}\nleq\mathbf{x}\ \forall\ \, \mathbf{y}\in S \}
\] 
of minimal elements of $S$ under the usual partial order $\leq$ is finite.
\end{Lemma}


\section{Arithmetical Structures on non-negative matrices}\label{ae}
In this section we recall the concepts of arithmetical structures and $M$-matrices.
Additionally we introduce a wider class of matrices that share some properties with $M$-matrices.
This new class of matrices, called quasi $M$-matrices, will be useful in the construction of the algorithm that computes the arithmetical structures on a matrix. Henceforth we assume that by ``matrix" we mean a square matrix of size $n$ for some positive integer $n$ unless the contrary is stated.

Given an integer non-negative matrix $L$ with diagonal zero, let 
\[
\mathcal{A}(L) = \Big\{ \left(\mathbf{d},\mathbf{r}\right)\in \mathbb{N}^n_+ \times \mathbb{N}^n_+ \, \Big| \, \left( \text{Diag}(\mathbf{d})-L\right) \mathbf{r}^t=\mathbf{0}^t \, \text{ and } \gcd(r_1,\ldots, r_n)=1\,  \Big\} 
\]
be the set of arithmetical structures on $L$.
Also, let
\[
\mathcal{D}(L)=\big\{ \mathbf{d}\in\mathbb{N}_+^n\, \big| \, (\mathbf{d},\mathbf{r})\in\mathcal{A}(L) \big\}\text{  and  }\mathcal{R}(L)=\big\{ \mathbf{r}\in\mathbb{N}_+^n\, \big|\, (\mathbf{d},\mathbf{r})\in\mathcal{A}(L) \big\},
\]
be the sets of \textit{d-arithmetical structures} and \textit{r-arithmetical structures} of $L$ respectively. 
As the next result shows it is not difficult to characterize when $\mathcal{A}(L)$ is non empty. 

\begin{Proposition}\label{empty}
If $L$ is a non-negative matrix with zero diagonal, then $\mathcal{A}(L) \neq \emptyset$ if and only if $L$ has no row with all entries equal to zero.
\end{Proposition}
\begin{proof}
$(\Rightarrow)$ If $\mathcal{A}(L)\neq \emptyset$, then there exists $\mathbf{d}, \mathbf{r}\in\mathbb{N}_+^n$ such that $[ \text{diag}(\mathbf{d})-L]\mathbf{r}^t=\mathbf{0}^t$. 
Thus $(L\bff{r}^t)_i=\mathbf{d}_i\mathbf{r}_i\geq 1$ for all $1\leq i \leq n$.
Moreover, since $L$ is integer non-negative and $\bff{r}\geq \bff{1}$, then $L\bff{1}^t\geq \bff{1}^t$. 
That is, $L$ has no row with all entries equal to zero.

$(\Leftarrow)$ Since $L$ is integer non-negative and has no row with all entries equal to zero, then $L\mathbf{1}^t\geq \bff{1}^t$.
Thus $(\textrm{Diag}(L\mathbf{1}^t)-L)\bff{1}^t=\bff{0}^t$ and therefore $(L\mathbf{1}^t,\mathbf{1})$ is an arithmetical structure of $L$.
\end{proof}

When $L\bff{1}^t\geq \bff{1}^t$ (that is, $L$ has all its rows different to $\bff{0}$), the arithmetical struture $(L\mathbf{1}^t,\mathbf{1})$ is called the canonical (or trivial) arithmetical structure of $L$.


\subsection{$M$-matrices}
A real matrix $A$ is a $Z$-matrix if $A_{ij}\leq 0$ for all $i\neq j$.
Also, a $Z$-matrix is an $M$-matrix if there exists a non-negative matrix $N$ and a non-negative number $\alpha$ such that $A = \alpha I- N$ with $\alpha$ greater or equal to the the spectral radius of $N$.
A more convenient characterization of $M$-matrices and non-singular $M$-matrices is given in terms of its principal minors.

\begin{Theorem}{\cite[Theorem 6.4.6 ($A_1$) and 6.2.3 ($A_1$)]{berman1994nonnegative}}
A $Z$-matrix is an $M$-matrix if and only if all of its principal minors are non-negative.
Moreover, an $M$-matrix is non-singular (that is, an $M$-matrix with determinant different from zero) if and only if its principal minors are positive.
\end{Theorem}

We recall that a principal minor of a matrix is the determinant of one of its sub-matrices obtained by eliminating the rows and columns indexed by a same subset.
Non-singular and singular matrices were studied in~\cite[Chapter 6]{berman1994nonnegative}.
$M$-matrices are present in a large variety of mathematical subjects, like numerical analysis, probability, economics, operations research, etc., see~\cite{berman1994nonnegative} and the references therein.
The next class of $M$-matrices were introduced in~\cite{arithmetical}.

\begin{Definition}
An $M$-matrix $M$ is an almost non-singular $M$-matrix if all of its proper principal minors are positive. 
\end{Definition}

That is, an $M$-matrix is an almost non-singular if and only if its minors are positive with the exception perhaps of its determinant which is non-negative.
Thus, $M$ is an almost non-singular $M$-matrix of size $n$ if and only if all of its proper sub-matrices of size strictly less than $n$ are non-singular $M$-matrices and $\det(M)\geq 0$. 
The next result relates arithmetical structures on a matrix and $M$-matrices.

\begin{Theorem}\cite[Theorem 3.2]{arithmetical}\label{almostrs}
If $M$ is a $Z$-matrix, then $M$ is an almost non-singular $M$-matrix with $\mathrm{det}(M)=0$ if and only if $M$ is irreducible and there is a vector $\bff{r}> \bff{0}$ such that $M\bff{r}^t=\bff{0}^t$.
\end{Theorem}

Therefore when $M$ is an irreducible $Z$-matrix, the concept of arithmetical structure is equivalent to that of almost non-singular $M$-matrix. 
A direct consequence of Theorem~\ref{almostrs} is the following result.

\begin{Corollary}\cite[Corollary 3.3]{arithmetical}\label{coroarithmetical}
If $M$ is an irreducible $Z$-matrix, then there exists $\mathbf{r}$ with all its entries positive such that $M\mathbf{r}^t= 0$ if and only if there exists $\mathbf{s}$ with all its entries positive such that $M^t\mathbf{s}^t = 0$.
\end{Corollary}

Thus Corollary~\ref{coroarithmetical} implies that if $L$ is a non-negative matrix with zero diagonal $L$, then $L$ and $L^t$ have the same set of $d$-arithmetical structures, but not necessarily the same set of $r$-arithmetical structures. 
Now, let us present the next properties of almost non-singular $M$-matrices.
	
\begin{Theorem}\cite[Theorem 2.6]{arithmetical}\label{almostequiv}
If $M$ is a real $Z$-matrix, then the following conditions are equivalent:
\begin{enumerate}
\item $M$ is an almost non-singular $M$-matrix.
\item $M+D$ is a non-singular $M$-matrix for any diagonal matrix $D>0$.
\item $\det(M)\geq 0$ and $\det(M+D)>\det(M+D^{'})>0$ for any diagonal matrices such that $D>D^{'}>0$.
\end{enumerate}
\end{Theorem}

The monotonicity of the determinant of an almost non-singular $M$-matrix is very important and motivates the concept of a quasi $M$-matrix, which is given next. 


\subsection{Quasi $M$-matrices}\label{quasi}
Here we will introduce the class of quasi $M$-matrices.
This class of matrices generalizes $M$-matrices in a very simple way.
Moreover, it has good properties that will be very useful for our algorithm.

\begin{Definition}
A real $Z$-matrix $M$ is called a quasi non-singular $M$-matrix if its proper principal minors are positive. 
And it is called a quasi $M$-matrix if all its proper principal minors are non-negative. 
\end{Definition}

Note that a quasi (non-singular) $M$-matrices are not necessarily $M$-matrices.
Moreover, the main difference between almost (non-singular) $M$-matrices and quasi (non-singular) $M$-matrices is that the determinant of the latter can be negative.
In other words, $M$ is a quasi (non-singular) $M$-matrix of size $n$ if all of its sub-matrices of size $n-1$ are (non-singular) $M$-matrices.
Quasi $M$-matrices are close to be quasi non-singular $M$-matrices in a similar way that $M$-matrices are from being non-singular $M$-matrices.
Thus the next result can be seen  as a generalization of~\cite[Lemma 4.1, Section 6]{berman1994nonnegative}. 

\begin{Theorem}
If $M$ is a real $Z$-matrix of size $n$ and $Id_n$ is the identity matrix of size $n$, then $M$ is a quasi $M$-matrix if and only if 
\[ 
M+\epsilon Id_n
\]
is a quasi non-singular $M$-matrix for any $\epsilon > 0$.
\end{Theorem}
\begin{proof}
$(\Rightarrow)$ Let $M_i$ be the submatrix resulting of deleting the $i^{th}$ row and column. 
Since $M$ is a quasi $M$-matrix we know that $M_i$ is an $M$-matrix for all $i$. 
Then, by \cite[Lemma 4.1, Section 6]{berman1994nonnegative}, every $M_i+\epsilon Id_{n-1}=(M+\epsilon Id_{n})_i$ is a non-singular $M$-matrix for any $\epsilon>0$.
Thus, for any positive $\epsilon$, the matrix $M+\epsilon Id_n$ is a $Z$-matrix with all of its proper sub-matrices non-singular $M$-matrices.
That is, $M+\epsilon Id_n$ is a quasi non-singular $M$-matrix for any $\epsilon>0$.

$(\Leftarrow)$ Conversely, if $M+\epsilon Id_n$ is a quasi non-singular $M$-matrix for all $\epsilon>0$, then $(M+\epsilon Id_n)_i$ is a non-singular $M$-matrix for all $i$ and $\epsilon>0$.
Thus, by~\cite[Lemma 4.1, Section 6]{berman1994nonnegative} $M_i$ is an $M$-matrix and therefore $M$ is a quasi $M$-matrix.
\end{proof}

Before continuing we will  fix some notation. 
If $M$ is a matrix, let $f_M(X)$ be the polynomial given by $\det(\textrm{Diag}(X)+M)$, where $X$ is the vector of variables $(x_1,x_2,\ldots,x_n)$. 

The next result is a key component of algorithms~\ref{A1} and~\ref{AA} given at next section and a generalization of Theorem~\ref{almostequiv} to quasi non-singular $M$-matrices.

\begin{Theorem}\label{monodet}
If $M$ is a real $Z$-matrix, then $M$ is a quasi non-singular $M$-matrix if and only if 
\[\det(M+D) > \det(M+D^{'}) > \det(M)\]
for every diagonal matrices such that $D>D^{'} > 0$.
\end{Theorem}
\begin{proof}
Let $D>D^{'}>0$ be diagonal matrices, $E_i=(e_{j,k})$ be the elementary matrix with $e_{i,i}=1$ and $e_{j,k}=0$ for all $(j,k)\neq (i,i)$ and $M_{\epsilon}=M+\epsilon E_i$.

$(\Rightarrow)$ Since $M+D$ is a quasi non-singular $M$-matrix, then
\[
\det(M_{\epsilon} [I;I])= \det(M[I;I])+ \epsilon \cdot \det(M[I \setminus i; I \setminus i]) > \det(M[I;I])
\]
for all $\epsilon>0$.
Thus, since $D=\sum_{i=1}^{n}d_i \cdot E_i$ for some $d_i \in \mathbb{R}_+$, 
\[
\det(M+D)>\det(M).
\]
Moreover, using similar arguments it can be proven that $\det(M+D^{'}+F) > \det(M+D^{'})$ for any diagonal matrix $F>0$. 
Finally, taking $F=D-D^{'}>0$ we get the result.

$(\Leftarrow)$
By hypothesis $f_M(X):=\det(\textrm{Diag}(X)+M)$ is a strictly increasing function on $(\mathbb{R}_+ \cup \{ 0\} )^n$ and we need to prove that $M$ is a quasi non-singular $M$-matrix, that is, $\det(M[J;J]) > 0$ for every $J\subsetneq [n]$. 
For all $J\subseteq [n]$, let 
\[
p_J(X_J):=f_M(X)|_{\{x_j=0\}_{j\notin J}}=\sum_{I \subseteq J} \det(M[I^c;I^c])\cdot x_{I} \text{ where } x_I=\prod_{i\in I}x_i \text{ for all } I\subseteq [n].
\]
We recall that $\det(M[\emptyset;\emptyset])=1$, $p_{[n]}=f_M(X)$ and $X_J=\{x_j: i\in J\}$.
Thus, if $J \neq \emptyset$, then $p_J$ is a strictly increasing function on $(\mathbb{R}_+\cup\{0\})^{J}$ with leading coefficient equal to $\det(M[J^c;J^c])$.
Now, if $p_J(x)=p_J(X_J)|_{\{x_j=x\}_{j\notin J}}$, then $dp_J(x)/dx|_{x=a}\geq 0$ for all $a\in \mathbb{R}_+\cup\{0\}$ and therefore its leading coefficient $\det(M[J^c;J^c])$ is non-negative for all $\emptyset \neq J\subseteq [n]$.
Thus, by~\cite[6.4.6 ($A_1$)]{berman1994nonnegative}, $M[J^c;J^c]$ is an $M$-matrix for all $\emptyset \neq J\subseteq [n]$.

Moreover, if $i\in [n]$ and $I=[n]\setminus i$, then
$p_i(x_i)=\det(M[I;I])x_i+\det(M)$ and therefore $\det(M[I;I])>0$ for all $i\in [n]$ otherwise $f_M(X)$ would not be strictly increasing.
Finally, since $M[I;I]$ is non-singular, by~\cite[6.2.3 ($A_1$)]{berman1994nonnegative} all its principal minors are positive.
\end{proof}
	

\section{The algorithm}\label{algo}
This section contains the main result of this article, an algorithm that computes all the arithmetical structures on a non-negative matrix with zero diagonal.
Before presenting the algorithm let us fix some notation.
Let $\mathbf{d}\in \mathbb{N}_+^n$, $X=(x_1,x_2,\ldots,x_n) $ and 
\[
f_{L,\mathbf{d}} (X)= \det (\mathnormal{Diag}(X+\mathbf{d})-L ).
\]
For simplicity we write $f_L (X)$ instead of $f_{L,\mathbf{0}} (X)$.
Now, let $\mathrm{coef}_{L,\mathbf{d}}(x^{a})$ be the coefficient of the monomial $\bff{x}^a=x_1^{a_1}\cdots\, x_n^{a_n}$ in $f_{L,\mathbf{d}}(X)$. 
The constant term of $f_{L,\mathbf{d}}(X)$ is equal to $\mathrm{coef}_{L,\mathbf{d}}(\bff{x}^{\bff{0}})=f_{L,\mathbf{d}}(\mathbf{0})$, which will be denoted by $c_{L,\mathbf{d}}$.
The coefficients of $f_{L,\mathbf{d}} (X)$ that are not the constant term are called \textit{nonconstant coefficients}.
Note that the coefficients of the polynomial $f_{L,\mathbf{d}}(X)$ are in correspondence with the principal minors of $\text{Diag}(\mathbf{d})-L$.
Thus, by inspecting the polynomial $f_{L,\mathbf{d}} (X)$ we can infer what type of $M$-matrix is $\text{Diag}(\mathbf{d})-L$. 
In a similar manner, $\mathbf{d}$ is a d-arithmetical structure of $L$ if and only if $c_{L,\mathbf{d}}=0$ and the nonconstant coefficients of $f_{L,\mathbf{d}}(X)$ are positive.
Note that the condition $c_{L,\mathbf{d}}=0$ is not enough to guarantee that a vector being a $d$-arithmetical structure, as the vector $r$ is not necessarily positive.

Recall that $\mathcal{D}_{\geq 0}(L)$ was defined as the set of $\bff{d}\in\mathbb{N}_+^n$ such that $\text{Diag}(\mathbf{d})-L$ is an almost non-singular $M$-matrix. 
Equivalently $\mathcal{D}_{\geq 0}(L)$ is the set of vectors $\mathbf{d}\in \mathbb{N}^n_+$ such that all nonconstant coefficients of $f_{L,\bff{d}}(X)$ are positive and $c_{L,\bff{d}}\geq 0$.
Thus, by Theorem~\ref{monodet} the problem of getting an almost non-singular $M$-matrix from a quasi non-singular$M$-matrix by adding a positive vector to the diagonal is similar to the knapsack problem, see for instance~\cite{mochila} for an extensive study of the knapsack problem. 

If $M$ is a quasi non-singular $M$-matrix, let 
\[
\mathcal{C}(M)=\{\mathbf{d}\in \mathbb{N}_+^n\, |\ (\textrm{Diag}(\bff{d})+M) \text{ is an almost non-singular } M\text{-matrix}\}.
\] 
It is not difficult to check that $\mathcal{C}(M)$ exists and is finite by Dickson's Lemma. 
Now let $\min \mathcal{D}_{\geq 0}(L)$ be the set of all minimal elements of $\mathcal{D}_{\geq 0}(L)$, $L_s$ is the submatrix of $L$ that results from removing the $s$-th row and column. 
Also, for any $\mathbf{d}\in \mathbb{Z}^{n-1}$ and $1\leq s\leq n$, let $\mathbf{d}^{(s)}\in \mathbb{Z}^{n}$ be given by
\begin{equation} \label{expand}
(\mathbf{d}^{(s)})_i=
\begin{cases}
\mathbf{d}_i & \text{if } 1\leq i< s,\\
1 & \text{    if } i=s,\\
\mathbf{d}_{i-1} & \text{ if } s < i \leq n.
\end{cases}
\end{equation}

Before presenting our first algorithm let us address the smaller case in the following result.
\begin{Lemma}\label{baselemma1}
If $a,b\in\mathbb{N}$, then
\begin{equation} \label{eq:bcase}
\min \mathcal{D}_{\geq 0}\left( \begin{matrix}
0 & a\\
b & 0
\end{matrix} \right)= \min \left\{ \Big( d,\max\big( 1,\Big\lceil \dfrac{ab}{d} \Big\rceil \big) \Big)\, \Big|\, d\in\mathbb{N}_+, d\leq \max(1,ab)  \right\},
\end{equation}
\end{Lemma}
\begin{proof}
This result is straightforward. 
Given a vector $(d_1,d_2)\in\mathbb{N}_+^2$, the only condition needed so that $(d_1,d_2)\in\min \mathcal{D}_{\geq 0}\Big( \begin{matrix}
0 & a\\
b & 0
\end{matrix} \Big)$, is that $d_1d_2\geq ab$.
\end{proof}
Note that this is the base case of the recursive algorithm that follows next.

\begin{Algorithm}\label{A1}
\mbox{}
\noindent\hrulefill 

\textbf{Input}: A non-negative square matrix $L$ of size $n$ with zero diagonal.

\textbf{Output}: $\min \mathcal{D}_{\geq 0}(L)$.

\begin{enumerate}[\hspace{1pt}(1)]
\item Compute $\tilde{A}_s=\min \mathcal{D}_{\geq 0}(L_s)$ for all $1\leq s\leq n$.
\vspace{0.7mm}
\item Let $A_s=\{\mathbf{\tilde{d}}^{(s)}\,|\,\mathbf{\tilde{d}}\in\,\tilde{A}_s\}$.
\vspace{1mm}
\item \textbf{For} $\boldsymbol{\delta}$ in $\prod_{s\in[n]}A_s$:
\vspace{1.4mm}
\item \hspace{3mm} $\mathbf{d}=\sup\big\{\boldsymbol{\delta}_1,\ldots, \boldsymbol{\delta}_n\big\}$.
\vspace{1.1mm}
\item \hspace{3mm} Let $S=\{s\ |\ \mathrm{coef}_{L,\mathbf{d}}(x_s) = 0 \}$
\item \hspace{3mm} $\textrm{Find}(L,\mathbf{d},S)$:
\item \hspace{6.5mm} \textbf{If} $|S| \geq 1$:
\item \hspace{10mm}
\textbf{For} $t\notin S$:
\item \hspace{13.5mm} Make $\mathbf{d}_t^{'}=\mathbf{d}_t+1$, $\mathbf{d}^{'}_r =\mathbf{d}_r$ for all $r\in [n]\setminus \{t\}$ and $\textrm{Find}(L,\mathbf{d}^{'},\emptyset)$.
\item \hspace{10mm} \textbf{If} $|S|\geq 2$:
\item \hspace{13.5mm}
\textbf{For} $s_1,s_2\in S$ ($s_1\neq s_2$):
\item \hspace{17mm}
Make $\mathbf{d}^{'}_{s_1}=\mathbf{d}_{s_1}+1$, $\mathbf{d}^{'}_{s_2}=\mathbf{d}_{s_2}+1$, $\bff{d}^{'}_r=\bff{d}_r$ for all $r\in [n]\setminus \{s_1,s_2\}$ and $\textrm{Find}(L,\mathbf{d}^{'},\emptyset)$. 

\item \hspace{6.5mm} \textbf{For} $\mathbf{d^{*}} \in \min \mathcal{C}(\mathrm{Diag}(\mathbf{d})-L)$: 
\item \hspace{10mm} ``Add" $\mathbf{d^{*}}+\mathbf{d}$ to $\min\mathcal{D}_{\geq 0}(L)$.
\item Return $\min \mathcal{D}_{\geq 0}(L)$.
\end{enumerate}
\noindent\hrulefill 
\end{Algorithm}

The vector at step (4) is the (unique) minimal vector greater or equal than $\boldsymbol{\delta}_i$ for every $i\in[n]$. At step $(6)$ we find all minimal vectors greater than $\mathbf{d}$ such that all coefficients of $f_{L,\mathbf{d}}(X)$ are positive except, maybe, for the constant term which is possibly zero.
The function ``add" at step $(14)$ means that we add the vector $\mathbf{d^{*}}+\mathbf{d}$ to the set $\min\mathcal{D}_{\geq 0}(L)$ whenever it is not greater than other vector already in the set.
Afterwards, by erasing every vector greater than $\mathbf{d^{*}}+\mathbf{d}$ from the set, the minimality of the set is assured.
 Now, we are ready to prove the correctness of the Algorithm~\ref{A1}.

\begin{Theorem}\label{correctA1}
Algorithm~\ref{A1} computes the set $\min\mathcal{D}_{\geq 0}(L)$ for any given non-negative matrix $L$ with zero diagonal.
\end{Theorem}	
\begin{proof}
We proceed by induction on the size of $L$. 
First, the case when $L$ is a matrix of size $2$ is solved by Equation (\ref{eq:bcase}).
Second, assume that the algorithm is correct for every matrix of size $1\leq m\leq n-1$ and let $L$ be a non-negative matrix with zero diagonal of size $n$.

Note that Algorithm~\ref{A1} can be split into three phases. 
The first phase, consisting of steps (1) to (4), creates by our induction hypothesis a set of vectors, say $\Delta$, such that for all $\bff{d}\in \Delta$ every proper principal minor of $\text{Diag}(\bff{d}) - L$ of size strictly less than $n - 1$ is positive. 
Moreover, every principal minor of $\text{Diag}(\bff{d}) - L$ of size $n-1$ is non-negative.

The second phase is encoded in steps (5) to (12), takes each vector $\mathbf{d}\in\Delta$ and increases some of their entries resulting in a new set of vectors, let us denote this set by $\Delta^{'}$. 
The vectors $\bff{d}^{'}\in\Delta^{'}$ are constructed such that every proper principal minors of $\text{Diag}(\bff{d}^{'}) - L$ are positive. 
For this, we begin at step (5) by identifying the indexes such that their corresponding minors of size $n-1$ are zero, that is, the set $S=\{s\ |\ \mathrm{coef}_{L,\mathbf{d}}(x_s) = 0 \}$. 
Note that if $S=\emptyset$ then $\mathbf{d}$ fulfills the desired property. 
Otherwise, if $|S|<n$ then by setting $\mathbf{d}^{'}_t=\mathbf{d}_t+1$ for some $t\notin S$ and $\mathbf{d}^{'}_r=\mathbf{d}_r$ for all $t\neq r\in [n]$, we have that the resulting vector $\bff{d}^{'}$ also fulfills the property because for every $s\in S$ the corresponding minor of size $n-1$ of $\text{Diag}(\bff{d}^{'}) - L$ is positive by Theorem~\ref{almostequiv}.
Also, if $|S|\geq 2$, as in step (11), let $s_1,s_2\in S$ and $s_1\neq s_2$. If we set $\mathbf{d}^{'}_{s_1}=\mathbf{d}_{s_1}+1$, $\mathbf{d}^{'}_{s_2}=\mathbf{d}_{s_2}+1$ and $\mathbf{d}^{'}_r=\mathbf{d}_r$ for all $ r\in [n]\setminus \{s_1,s_2\}$. 
Then, similarly, we have that the proper principal minors of $\text{Diag}(\bff{d}^{'}) - L$ are positive by Theorem~\ref{almostequiv}.

Finally, we will prove that in steps (13), (14) and (15), the algorithm increases the vectors in $\Delta^{'}$ further so that we get all the almost non-singular M-matrices in $\min\mathcal{D}_{\geq 0} (L)$.
Let us note that if $\mathbf{d}^{'}$ is a vector given by steps (1)-(12), then by both Theorem~\ref{monodet} and the definition of $\mathcal{C}(M)$, every vector $\mathbf{f}\geq\mathbf{d}^{'}$ such that $(\text{Diag}(\mathbf{f})-L)$ is an almost non-singular $M$-matrix will be reached on steps $(13)$ and $(14)$. 
Therefore we only need to prove that every $\mathbf{f}\in\min\mathcal{D}_{\geq 0}(L)$ is reached by some vector constructed in steps (1)-(12).

In order to prove this, for every $\mathbf{f}\in\mathcal{D}_{\geq 0}(L)$, let $\mathbf{f}_{|s}$ be the vector equal to $\mathbf{f}$ without the $s$-th entry.
That is,
\[
(\mathbf{f}_{|s})_i=\begin{cases}
\mathbf{f}_i, & \text{ if }1\leq i\leq s-1,\\
\mathbf{f}_{i+1}, & \text{ if }s\leq i\leq n-1.
\end{cases}
\]
Then for every $s\in[n]$, we have that $\mathbf{f}_{|s}\in\mathcal{D}_{\geq 0}(L_s)$ and there exists $\mathbf{\tilde{f}}\in\min\mathcal{D}_{\geq 0}(L_s)$ such that $\mathbf{\tilde{f}}\leq \mathbf{f}_{|s}$. 
Consequently, we have that 
\[
\max_{s\in [n]} \left\{ (\mathbf{\tilde{f}}^{(s)})_i \right\} \leq \mathbf{f}_i,
\]
where $\mathbf{f}^{(s)}$ is as in equation (\ref{expand}). In other words, every $\mathbf{f}\in\min\mathcal{D}_{\geq 0}(L)$ is greater or equal than a vector presented by step (4). 
Therefore let $\mathbf{f}\in\min\mathcal{D}_{\geq 0}(L)$ and let $\bff{d}\leq \bff{f}$ be such vector given at step (4). 
Then, assume that there is no vector $\bff{d}^{'}\geq \bff{d}$ given by steps (5)-(12) such that $\bff{f}\geq \bff{d}^{'}$. Note that $S=\{s\ |\ \mathrm{coef}_{L,\mathbf{d}}(x_s) = 0 \}\neq \emptyset$ and that any vector given by steps (5)-(12) can not be greater or equal than $\bff{f}$. 
Thus $\bff{f}=\bff{d}+a\bff{e_s}$ for some $a\in \mathbb{N}_+$ and some $s\in S$, where $\bff{e_s}\in \mathbb{N}^n$ is the standard unit vector with its $s$-th entry equal to $1$. 
Therefore the minor of $\text{Diag}(\bff{f})-L$ of size $n-1$ corresponding to removing the $s$-th row and column is zero, a contradiction since $\text{Diag}(\bff{f})-L$ is an almost non-singular $M$-matrix. 
Hence concluding that, indeed, there is a vector $\bff{d}^{'}\geq \bff{d}$ given by steps (5)-(12) such that $\bff{f}\geq \bff{d}^{'}$. 
\end{proof}

Note that Algorithm~\ref{A1} is not of polynomial-time because, in general, the extended knapsack problem of finding $\mathcal{C}(M)$ is not of polynomial-time. 
On the other hand, thanks to the recursive structure of the algorithm, we get rid of the need of checking the value of the $2^n-2$ proper principal minors of $L$. Now, we present the following algorithm that uses Algorithm~\ref{A1}.
	
\begin{Algorithm}\label{AA}
\mbox{}
\noindent\hrulefill 

\textbf{Input}: A non-negative integer square matrix $L$ with zero diagonal.

\textbf{Output}: $\mathcal{D}(L)$.

\begin{enumerate}[\hspace{1pt}(1)]
\item \textbf{If} (L is irreducible):
\item \hspace{3mm} $A=\min \mathcal{D}_{\geq 0}(L)$,
\item \hspace{3mm} $D=\{\mathbf{d}\in A: f_{L,{\mathbf{d}}}(\mathbf{0})=0 \}$,
\item \hspace{3mm} Return D.
\item \textbf{Elif} (L has a row equal to $\mathbf{0}$):
\item \hspace{3mm} Return $\emptyset$.
\item \textbf{Else}:
\item \hspace{3mm} $\text{ 'there is an infinite number of arithmetical structures'}$, (see Theorem~\ref{finiteness}).
\end{enumerate}
\noindent\hrulefill 
\end{Algorithm}
	
Note that if we have a $d$-arithmetical structure on a matrix $L$, then it is very simple to get the corresponding $r$-arithmetical structure. 
We only need to compute the kernel of $\textrm{Diag}(\mathbf{d})-L$. 
Now, we are able to present the correctness of Algorithm~\ref{AA}.

\begin{Corollary}\label{correctAA}
Algorithm~\ref{AA} computes the set of arithmetical structures on any non-negative integer square matrix $L$ with diagonal zero.
\end{Corollary}
\begin{proof}
It follows directly from Proposition~\ref{empty} and Theorems~\ref{finiteness},~\ref{almostrs} and~\ref{correctA1}.
\end{proof}

The next example illustrates how Algorithm~\ref{A1} works.
Moreover, it will give us a glance of its complexity.

\begin{Example}
Let $G$ be the graph given in Figure~\ref{fig:01} and $L$ be its adjacency matrix.

\vspace{-10mm}

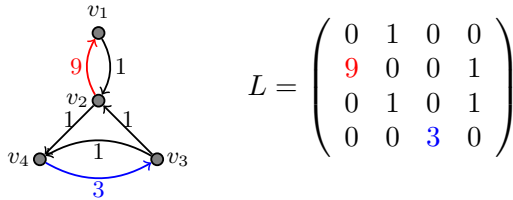
\begin{figure}[h]

\hspace{1.88cm}
\begin{tabular}{cc}
\multirow{9}{3cm}{
\begin{tikzpicture}[scale=0.6, line width=0.7pt]
\tikzstyle{every node}=[minimum width=4.5pt, inner sep=0pt, circle]
\draw (-0.8,-1) node (v4) [draw, fill=gray, label=left:{\footnotesize $v_4$}] {};
\draw (1.8,-1) node (v3) [draw, fill=gray, label=right:{\footnotesize $v_3$}] {};
\draw (0.5,1.8) node (v1) [draw, fill=gray, label=above:{\footnotesize $v_1$}] {};
\draw (0.5,0.3) node (v2) [draw, fill=gray, label=left:{\footnotesize $v_2$}] {};
\draw [->, color=blue] (v4) to [out=60,in=120, bend right]node[below]{\textcolor{blue}{\footnotesize $3$}} (v3);
\draw [->] (v3) to [out=60,in=120, bend right]node[below]{\footnotesize $1$} (v4);
\draw [->] (v3) to node[above]{\footnotesize $1$} (v2);
\draw [->] (v2) to node[above]{\footnotesize $1$} (v4);
\draw [->, color=red] (v2) to [out=60,in=120, bend left]node[left]{\textcolor{red}{\footnotesize $9$}} (v1);
\draw [->] (v1) to [out=60,in=120, bend left]node[right]{\footnotesize $1$} (v2);
\end{tikzpicture}
}
& \\
& \\
&
$L=\left(\begin{array}{ccccc}
0 & 1 & 0 & 0\\
\color{red}{9} & 0 & 0 & 1\\
0 & 1 & 0 & 1\\
0 & 0 & \color{blue}{3} & 0
\end{array}\right)$\\
& \\
\end{tabular}
	
\caption{A digraph $G$ and its adjacency matrix $L$.}\label{fig:01}
\end{figure}
\vspace{-5mm}
Since $G$ is a strongly connected graph, then $L$ is an irreducible matrix. 
Therefore, the first step of Algorithm~\ref{AA} consists of computing the sets $\min \mathcal{D}_{\geq 0}(L_s)$ for all the sub-matrices of $L$ of size $n-1$.
It can be checked that 
\[\min\mathcal{D}_{\geq 0}(L_1)=\{ (1,2,3),(1,3,2),(2,1,5),(2,5,1),(1,1,6),(1,6,1),(3,1,4),(3,4,1),(3,2,2) \},\] 
$$\min\mathcal{D}_{\geq 0}(L_2)=\{ (1,2,2),(1,1,4),(1,4,1) \}$$
and $\min\mathcal{D}_{\geq 0}(L_3)=\min\mathcal{D}_{\geq 0}(L_4)= \{ (2,5,1),(5,2,1),(3,4,1),(4,3,1),(1,10,1),(10,1,1) \}.$
From this we get that the set of vectors $\bff{d}$ given at steps $(1)-(4)$ of Algorithm~\ref{A1} is equal to 
\[
\begin{Bmatrix}
(5,2,2,3)_{-12},&(5,2,3,2)_{-12},&(10,1,2,3)_{-27},&(10,1,3,2)_{-27},&(2,5,2,2)_{-5},\\
(3,4,2,2)_{-6},&(4,3,2,2)_{-9},&(1,10,2,2)_{-2},&(5,2,1,5)_{-13},&(5,2,5,1)_{-13},\\
(10,1,1,6)_{-27},&(10,1,6,1)_{-27},&(1,10,1,4)_{-2},&(1,10,4,1)_{-2},&(2,5,1,4)_{-5},\\
(2,5,4,1)_{-2},&(3,4,1,4)_{-6},&(3,4,4,1)_{-6},&(4,3,1,4)_{-9},&(4,3,4,1)_{-9}.
\end{Bmatrix}
\]
where the sub-index in each vector $\bff{d}$ corresponds to the determinant of $(\textrm{Diag}(\bff{d})-L)$.
Following the rest of Algorithm~\ref{A1} we get that $\min\mathcal{D}_{\geq 0}(A(G))$ is equal to
\begin{center}
\begin{tabular}{|cccccc|}
\hline
\vspace{-2mm}
&&&&&\\
$(5,3,2,3)_3$&$(5,3,3,2)_3$&$(9,2,2,3)_0$&$(9,2,3,2)_0$&$(6,2,3,3)_0$&$(6,2,2,5)_3$\\
$(6,2,5,2)_3$&$(5,2,6,3)_0$&$(5,2,3,6)_0$&$(5,2,2,9)_0$&$(5,2,9,2)_0$&$(5,2,5,4)_2$\\
$(5,2,4,5)_2$&$(7,2,2,4)_4$&$(7,2,4,2)_4$&$(5,5,2,2)_1$&$(2,8,2,2)_1$&$(2,5,5,2)_1$\\
$(2,5,2,5)_1$&$(3,6,2,2)_0$&$(2,6,3,2)_3$&$(2,6,2,3)_3$&$(2,5,3,3)_0$&$(3,4,2,3)_0$\\
$(3,4,3,2)_0$&$(9,4,2,2)_0$&$(1,10,3,2)_0$&$(1,10,2,3)_0$&$(1,12,2,2)_0$&$(18,1,3,3)_0$\\
$(10,1,11,3)_0$&$(10,1,3,11)_0$&$(14,1,4,3)_3$&$(14,1,3,4)_3$&$(12,1,5,3)_0$&$(12,1,3,5)_0$\\
$(12,1,4,4)_3$&$(11,1,7,3)_3$&$(11,1,3,7)_3$&$(11,1,5,4)_1$&$(11,1,4,5)_1$&$(10,1,9,4)_3$\\
$(10,1,4,9)_3$&$(10,1,7,5)_2$&$(10,1,5,7)_2$&$(10,1,6,6)_3$&$(23,1,2,4)_1$&$(23,1,4,2)_1$\\
$(16,1,2,5)_1$&$(16,1,5,2)_1$&$(14,1,2,6)_3$&$(14,1,6,2)_3$&$(12,1,2,8)_3$&$(12,1,8,2)_3$\\
$(11,1,2,10)_1$&$(11,1,10,2)_1$&$(10,1,2,17)_1$&$(10,1,17,2)_1$&$(4,3,4,2)_3$&$(4,3,2,4)_3$\\
$(4,3,3,3)_6$&$(18,2,1,5)_0$&$(18,2,5,1)_0$&$(5,2,1,18)_0$&$(5,2,18,1)_0$&$(6,3,1,5)_0$\\
$(6,3,5,1)_0$&$(6,2,1,9)_0$&$(6,2,9,1)_0$&$(7,2,1,8)_4$&$(7,2,8,1)_4$&$(8,2,1,7)_4$\\
$(8,2,7,1)_4$&$(5,3,1,6)_3$&$(5,3,6,1)_3$&$(9,2,1,6)_0$&$(9,2,6,1)_0$&$(4,4,1,5)_2$\\
$(4,4,5,1)_2$&$(10,2,1,6)_3$&$(10,2,6,1)_3$&$(10,1,1,33)_0$&$(10,1,33,1)_0$&$(11,1,1,20)_1$\\
$(11,1,20,1)_1$&$(12,1,1,15)_0$&$(12,1,15,1)_0$&$(13,1,1,13)_1$&$(13,1,13,1)_1$&$(14,1,1,12)_3$\\
$(14,1,12,1)_3$&$(15,1,1,11)_3$&$(15,1,11,1)_3$&$(16,1,1,10)_1$&$(16,1,10,1)_1$&$(18,1,1,9)_0$\\
$(18,1,9,1)_0$&$(23,1,1,8)_1$&$(23,1,8,1)_1$&$(36,1,1,7)_0$&$(36,1,7,1)_0$&$(1,12,1,4)_0$\\
$(1,12,4,1)_0$&$(1,11,1,5)_1$&$(1,11,5,1)_1$&$(1,10,1,6)_0$&$(1,10,6,1)_0$&$(5,5,1,4)_1$\\
$(5,5,4,1)_1$&$(3,6,1,4)_0$&$(3,6,4,1)_0$&$(3,5,1,5)_5$&$(3,5,5,1)_5$&$(2,8,1,4)_1$\\
$(2,8,4,1)_1$&$(2,6,1,5)_0$&$(2,6,5,1)_0$&$(2,5,1,9)_0$&$(2,5,9,1)_0$&$(9,4,1,4)_0$\\
$(9,4,4,1)_0$&$(3,4,1,6)_0$&$(3,4,6,1)_0$&$(4,3,1,7)_0$&$(4,3,7,1)_0$&\\
\hline
\end{tabular}
\end{center}
\vspace{2mm}

Thus, $\min\mathcal{D}_{\geq 0}(L)$ has $125$ elements and $54$ of them are d-arithmetical structures on $L$. The set $\mathcal{D}(L)$ is listed below. 

\begin{center}
\begin{tabular}{|cccccc|}
\hline	
\vspace{-2mm}
&&&&&\\
$(9,2,2,3)$&$(9,2,3,2)$&$(6,2,3,3)$&$(5,2,6,3)$&$(5,2,3,6)$&$(5,2,2,9)$\\
$(5,2,9,2)$&$(3,6,2,2)$&$(2,5,3,3)$&$(3,4,2,3)$&$(3,4,3,2)$&$(9,4,2,2)$\\
$(1,10,3,2)$&$(1,10,2,3)$&$(1,12,2,2)$&$(18,1,3,3)$&$(10,1,11,3)$&$(10,1,3,11)$\\
$(12,1,5,3)$&$(12,1,3,5)$&$(18,2,1,5)$&$(18,2,5,1)$&$(5,2,1,18)$&$(5,2,18,1)$\\
$(6,3,1,5)$&$(6,3,5,1)$&$(6,2,1,9)$&$(6,2,9,1)$&$(9,2,1,6)$&$(9,2,6,1)$\\
$(10,1,1,33)$&$(10,1,33,1)$&$(12,1,1,15)$&$(12,1,15,1)$&$(18,1,1,9)$&$(18,1,9,1)$\\
$(36,1,1,7)$&$(36,1,7,1)$&$(1,12,1,4)$&$(1,12,4,1)$&$(1,10,1,6)$&$(1,10,6,1)$\\
$(3,6,1,4)$&$(3,6,4,1)$&$(2,6,1,5)$&$(2,6,5,1)$&$(2,5,1,9)$&$(2,5,9,1)$\\
$(9,4,1,4)$&$(9,4,4,1)$&$(3,4,1,6)$&$(3,4,6,1)$&$(4,3,1,7)$&$(4,3,7,1)$\\
\hline \end{tabular}
\end{center}
\end{Example}

At this point, given an integer square matrix $L$ with diagonal zero we have an algorithm that computes some of the integer solutions of the polynomial equation $f_L(X)=\mathrm{det}(Diag(X)-L)=0$. 
Thus, it is natural to ask: When is a polynomial is the determinant of a matrix with diagonal $X$? 
In the following we will see that not every polynomial is the determinant of a matrix of the form 
\[
Diag(X)-L.
\] 
Moreover, we show that the set of monic (every term) square-free polynomials that are the determinant of a matrix $Diag(X)-L$ is in some sense very small.
Firstly, it is clear that if a polynomial is equal to $\mathrm{det}(Diag(X)-L)$, then it must be monic (the coefficient of the monomial $x_1x_2\dots x_n$ is always $1$) and (every term) square-free. 
Therefore we restrict to monic polynomials with every term square-free. 
For the rest of this section we will simply call them monic and square-free polynomials. 
Let
\[
\mathbb{Z}[X]^*=\big\{ f\in\mathbb{Z}[X]\,\big|\,f\text{ is monic and square-free} \big\}/(\sim),
\]
where $f\sim g$ if there exists $\bff{d}\in\mathbb{Z}^n$ such that $f(X+\bff{d})=g(X)$.
We consider two polynomials equivalent because when one can be obtained as an evaluation of the other is because its integer solutions are essentially the same.
Note that 
\[
\mathbb{Z}[X]^*\simeq\big\{ f\in\mathbb{Z}[X]\,\big|\,f\text{ is monic square-free with } \mathrm{coef} \left( \frac{\prod_{i=1}^{n} x_i}{x_j}\right)=0 \text{ for all } 1\leq j \leq n\big\}. 
\]
Also, let $\mathbb{MP}[X]=\big\{f\in\mathbb{Z}[X]^*\,\big|\,f=\det(\textrm{Diag}(X)-L)\text{ for some matrix } L \text{ with zero diagonal} \big\}$.
Clearly 
\[
\mathbb{MP}[X]\subseteq\mathbb{Z}[X]^*.
\] 
If $|X|=2$, we can easily check equality holds, that is, $\mathbb{MP}[x_1,x_2]=\mathbb{Z}[x_1,x_2]^*$.
In a similar way, for $|X|=3$ we will prove the following result.

\begin{Proposition}\label{xsizethree}
If $X=(x_1,x_2,x_3)$, then
\[
\mathbb{MP}[X]=\big\{ f=x_1x_2x_3+a_1x_1+a_2x_2+a_3x_3+b\in\mathbb{Z}[X]^*\,\big|\,b=\dfrac{a_1a_2a_3-n^2}{n}\in \mathbb{Z}\big\}.
\]
\end{Proposition}
\begin{proof}
Let $A=\begin{pmatrix}
0 & a_{1,2} & a_{1.3}\\
a_{2.1} & 0 & a_{2,3}\\
a_{3,1} & a_{3,2} & 0
\end{pmatrix}$ be an integer matrix with zero diagonal. 
Then 
\[
\det(Diag(x_1,x_2,x_3)-A)=x_1x_2x_3-a_{2,3}a_{3,2}x_1-a_{1,3}a_{3,1}x_2-a_{1,2}a_{2,1}x_3+(-a_{1,3}a_{2,1}a_{3,2}-a_{1,2}a_{2,3}a_{3,1})
\]
Now, let us set $a_1=-a_{2,3}a_{3,2}$, $a_2=-a_{1,3}a_{3,1}$, $a_3=-a_{2,3}a_{3,2}$ and $b=(-a_{1,3}a_{2,1}a_{3,2}-a_{1,2}a_{2,3}a_{3,1})$. 
Therefore, $b=\dfrac{a_1a_2a_3}{a_{1,2}a_{2,3}a_{3,1}}-a_{1,2}a_{2,3}a_{3,1}$. We know that both $b$ and $a_{1,2}a_{2,3}a_{3,1}$ are integer numbers. 
Then $a_{1,2}a_{2,3}a_{3,1}$ is a divisor of $a_1a_2a_3$, since $\dfrac{a_1a_2a_3}{a_{1,2}a_{2,3}a_{3,1}}$ is also an integer. 
Now, we can conclude that $b=\dfrac{a_1a_2a_3}{n}-n$, where $n$ divides $a_1a_2a_3$.  
\end{proof}

Note that $b\in \mathbb{Z}$ if and only if $n$ divides $a_1a_2a_3$.
If $a_1a_2a_3\neq 0$, then there is a finite set of $b^{'}$s on $\mathbb{Z}$ such that $f=x_1x_2x_3+a_1x_1+a_2x_2+a_3x_3+b$ is in $\mathbb{MP}[x_1,x_2,x_3]$ and therefore $\mathbb{MP}[X]\subsetneq\mathbb{Z}[X]^*$ for all $n\geq 3$.
On the other hand, if $a_1a_2a_3=0$, then $x_1x_2x_3+a_1x_1+a_2x_2+a_3x_3+b$ is in $\mathbb{MP}[x_1,x_2,x_3]$ for every $b\in\mathbb{Z}$. 
Furthermore, if $f$ comes from a matrix, then 
\[
f=\det \begin{bmatrix}
x_1 & -n_3 & \frac{a_2}{n_2}\\
\frac{a_3}{n_3} & x_2 & -n_1\\
-n_2 & \frac{a_1}{n_1} & x_3 
\end{bmatrix},
\]
where $n=n_1n_2n_3$ and $n_i|a_i$ (here we are considering every integer as a ``divisor" of $0$).

Next example will be helpful to illustrate this.
\begin{Example}\label{polynomial_not_determinant}
If $g=x_1x_2x_3-19x_1+2x_2+3x_3+b$, then
\[
b=\frac{-114}{n}-n\text{ where } n\in\mathrm{Div}(114)=\pm\{1,2,3,6,19,38,57,114 \}.
\]
This implies that $b\in \pm\{25,41,59,115 \}$.
It is not difficult to check by Proposition~\ref{xsizethree} that $f(x_1,x_2,x_3)=x_1x_2x_3-19x_1+2x_2+3x_3-23\notin\mathbb{MP}[x_1,x_2,x_3]$.
\end{Example}

When $n\geq 4$ we have similar restrictions for the coefficients of the polynomial $f$.
Moreover, the gap between $\mathbb{MP}[X]$ and $\mathbb{Z}[X]^*$ grows as $n$ grows.
	
We finish this section by presenting computational data about arithmetical structures of graphs with less or equal to five vertices.
Additionally, this data provides evidence for conjecture~\ref{conj}.

\begin{Conjecture}\cite[Conjecture 6.10]{arithmetical}\label{conj}
If $G$ is a simple graph with $n$ vertices, then 
\[
\Big|\,\mathcal{A}(P_n)\,\Big|\leq \Big|\,\mathcal{A}(G)\,\Big| \leq \Big|\,\mathcal{A}(K_n)\,\Big|,
\]
where $P_n$ and $K_n$ are the path with and the complete graph with $n$ vertices respectively.
\end{Conjecture}

In simple words, Conjecture~\ref{conj}  says that for graphs, arithmetical structures are most simple when it is a path and the most complicated case happens when it is the complete graph.
Moreover, when $G$ is the star graph with $n$ leaves, then its arithmetical structures satisfy that $\sum_{i=1}^n \frac{1}{d_i}\in \mathbb{N}_+$.
Thus the arithmetical structures on the star with $n+1$ vertices are more complicated than the arithmetical structures on the complete graph with $n$ vertices.
In general, for many graphs with $n+1$ vertices, their arithmetical structures are at least as complicated as the arithmetical structures on $K_n$.

There exists upper bounds for the entries of vector $\bff{r}$ of an arithmetical structure.
For instance, when $G$ is a graph with $n$ vertices and $e(G)$ its number of edges, \cite[Theorem 3.4]{bound} establish that
\[
\bff{r}\leq \frac{1}{(n-1)!}e(G)^{3\cdot2^{n-2}-2} \bff{1}.
\]
This upper bound will lead to a brute force algorithm with time complexity of the order of \[\frac{n^2}{((n-1)!)^n}e(G)^{n(3\cdot 2^{n-2}-2)} \geq n^2 (n-1)^{3n\cdot 2^{n-2} -2n-n^2} \text{ for any connected graph }G\]

Moreover, this brute-force algorithm applied to the complete graph $K_n$ would have a complexity of the order of 
\[
n^{3n\cdot 2^{n-2} -2n+ 2} \cdot \dfrac{(n-1)^{n(3\cdot 2^{n-2} -2)}}{ ((n-1)!)^n }.
\]
A deeper study of the possible value of the largest entry of an arithmetical structure on the complete graph is conducted in~\cite{harris}.
In order to have an idea of the complexity of arithmetical structures on the complete graph see the sequence A002967~\cite{ef1}. 
It shows that 
\[
\Big| \mathcal{A}(K_6)\Big|=2025462, \Big| \mathcal{A}(K_7)\Big|= 1351857641\text{ and }\Big| \mathcal{A}(K_8)\Big|\simeq6.25\times 10^{12}.
\] 
On the other hand, the largest $d_i$ such that $\mathbf{d}\in\mathcal{D}(K_n)$ is given by the Sylvester's sequence $a(n)=a(n-1)(a(n-1)-1)+1$ with $a(0)=2$. 
For instance the highest $d_i$ in a d-arithmetical structure of $K_8$ is about $1.13\times 10^{26}$. 
Equivalently $\max\{d_i\,|\, i=1,\ldots,n \}=\lfloor C^{2^n}+\frac{1}{2} \rfloor$, where $C \simeq 1.2640847353$ is the Vardi constant.
Thus a brute force algorithm based on this upper bound for the entries of the d-arithmetical structure of the complete graph is of the order of $n^6 C^{n\cdot 2^n}$.

The reader can find a code for Algorithm~\ref{A1}, written in sagemath~\cite{Sagemath}, at the link in~\cite{Github}. 
Note that the implementation of the algorithm still can be improved. 
For instance, in the code given above we are using only the minimal elements in the set defined by steps (1) to (4). 
For the case of simple graphs, we can further improve the computation using the symmetry of the graph. 
More precisely, the twin vertices of the graph. 
Let $G$ be a graph, a pair of vertices $w,v\in V(G)$ are said to be twins if $N_G(w)\setminus \{v\} = N_G(v)\setminus \{w\}$. 
Now, assume $G$ is a simple graph and $w,v\in V(G)$ are twins. If $\bff{f}$ is a d-arithmetical structure of $G-w$. 
Then, for $G-v$, there is a d-arithmetical structure $\bff{h}$, such that $\bff{h}_u = \bff{f}_ u$ for every $w\neq u\neq v$ and $\bff{h}_w = \bff{f}_v$.
Hence, we can relax step (3) of Algorithm~\ref{A1} by removing the two sets corresponding to a pair of twin vertices and adding instead a similar set of the same size. 
A similar observation can be made for larger sets of twin vertices. A code for simple graphs of five vertices implementing this changes can also be found in \cite{Github}.
Moreover, we can proceed similarly for multidigraphs if all the weights of all the corresponding edges of the twin vertices in question are equal. 
In Table \ref{tab:graphs1} we list all connected graphs with three and four vertices together with the number of arithmetical graphs and the number of elements in $\min \mathcal{D}_{\geq 0}(A(G))$. 
As well as for the seven graphs on $5$ vertices with fewer elements in $\min \mathcal{D}_{\geq 0}(L)$, which are also the fastest to compute. 

\begin{table}[H]
\begin{center}
\begin{tabular}{|c|c|c|c|c|c|c|c|c|}
\hline
Graph & {\small $\big|\mathcal{A}\big|$} & {\small $\big|\min \mathcal{D}_{\geq 0}\big|$} & Graph & {\small $\big|\mathcal{A}\big|$} & {\small $\big|\min \mathcal{D}_{\geq 0}\big|$} & Graph & {\small $\big|\mathcal{A}\big|$} & {\small $\big|\min \mathcal{D}_{\geq 0}\big|$}\\
\hline
&&&&&&&&\\
\begin{tikzpicture}[scale=1, line width=0.7pt]
\tikzstyle{every node}=[minimum width=4.5pt, inner sep=0pt, circle]
\draw (-1,0) node (v1) [draw, fill=gray, label=left:{\notag}] {};
\draw (0,0) node (v2) [draw, fill=gray, label=right:{\notag}] {};
\draw (-0.5,0.75) node (v3) [draw, fill=gray, label=right:{\notag}] {};
\draw (v2) -- (v3) -- (v1);
\end{tikzpicture} & 
\begin{tabular}{c}
$2$\\
 \\
\end{tabular}& 
\begin{tabular}{c}
$2$\\
\\
\end{tabular}
&
\begin{tikzpicture}[scale=1, line width=0.7pt]
\tikzstyle{every node}=[minimum width=4.5pt, inner sep=0pt, circle]
\draw (-1,0) node (v1) [draw, fill=gray, label=left:{\notag}] {};
\draw (0,0) node (v2) [draw, fill=gray, label=right:{\notag}] {};
\draw (-0.5,0.75) node (v3) [draw, fill=gray, label=right:{\notag}] {};
\draw (v1) -- (v2) -- (v3) -- (v1);
\end{tikzpicture} &\begin{tabular}{c}
$10$\\
\\
\end{tabular}&\begin{tabular}{c}
$10$\\
\\
\end{tabular}&
\begin{tikzpicture}[scale=0.8, line width=0.7pt]
\tikzstyle{every node}=[minimum width=4.5pt, inner sep=0pt, circle]
\draw (-1,0) node (v1) [draw, fill=gray, label=left:{\notag}] {};
\draw (0,0) node (v2) [draw, fill=gray, label=right:{\notag}] {};
\draw (0,1) node (v3) [draw, fill=gray, label=right:{\notag}] {};
\draw (-1,1) node (v4) [draw, fill=gray, label=left:{\notag}] {};
\draw (v1) -- (v2) -- (v3) -- (v4);
\end{tikzpicture} &
\begin{tabular}{c}
$5$\\
\\
\end{tabular} &
\begin{tabular}{c}
$5$\\
\\
\end{tabular}
\\
\hline
\begin{tikzpicture}[scale=0.8, line width=0.7pt]
\tikzstyle{every node}=[minimum width=4.5pt, inner sep=0pt, circle]
\draw (0,0) node (v1) [draw, fill=gray, label=below:{\notag}] {};
\draw (-0.75,-0.5) node (v2) [draw, fill=gray, label=left:{\notag}] {};
\draw (0.75,-0.5) node (v3) [draw, fill=gray, label=right:{\notag}] {};
\draw (0,0.8) node (v4) [draw, fill=gray, label=left:{\notag}] {};
\draw (v1) -- (v2); 
\draw (v1) -- (v3);
\draw (v1) -- (v4);
\end{tikzpicture} & 
\begin{tabular}{c}
$14$\\
\\
\end{tabular}&
\begin{tabular}{c}
$14$\\
\\
\end{tabular} &
\begin{tikzpicture}[scale=0.8, line width=0.7pt]
\tikzstyle{every node}=[minimum width=4.5pt, inner sep=0pt, circle]
\draw (-1.1,-0.6) node (v1) [draw, fill=gray, label=left:{\notag}] {};
\draw (0.1,-0.6) node (v2) [draw, fill=gray, label=right:{\notag}] {};
\draw (-0.5,0.7) node (v3) [draw, fill=gray, label=left:{\notag}] {};
\draw (-0.5,0.0) node (v4) [draw, fill=gray, label=right:{\notag}] {};
\draw (v1) -- (v2) -- (v4) -- (v1);
\draw (v4) -- (v3);
\end{tikzpicture} &\begin{tabular}{c}
$26$\\
\\
\end{tabular}&\begin{tabular}{c}
$42$\\
\\
\end{tabular}&
\begin{tikzpicture}[scale=0.8, line width=0.7pt]
\tikzstyle{every node}=[minimum width=4.5pt, inner sep=0pt, circle]
\draw (-1,0) node (v1) [draw, fill=gray, label=left:{\notag}] {};
\draw (0,0) node (v2) [draw, fill=gray, label=right:{\notag}] {};
\draw (0,1) node (v3) [draw, fill=gray, label=right:{\notag}] {};
\draw (-1,1) node (v4) [draw, fill=gray, label=left:{\notag}] {};
\draw (v1) -- (v2) -- (v3) -- (v4) -- (v1);
\end{tikzpicture} & 
\begin{tabular}{c}
$35$\\
\\
\end{tabular}&
\begin{tabular}{c}
$35$\\
\\
\end{tabular}
\\
\hline
 \begin{tikzpicture}[scale=0.8, line width=0.7pt]
\tikzstyle{every node}=[minimum width=4.5pt, inner sep=0pt, circle]
\draw (-1,0) node (v1) [draw, fill=gray, label=left:{\notag}] {};
\draw (0,0) node (v2) [draw, fill=gray, label=right:{\notag}] {};
\draw (0,1) node (v3) [draw, fill=gray, label=right:{\notag}] {};
\draw (-1,1) node (v4) [draw, fill=gray, label=left:{\notag}] {};
\draw (v1) -- (v2) -- (v3) -- (v4) -- (v1);
\draw (v1) -- (v3);
\end{tikzpicture} &\begin{tabular}{c}
$63$\\
\\
\end{tabular}&\begin{tabular}{c}
$137$\\
\\
\end{tabular} &
\begin{tikzpicture}[scale=0.8, line width=0.7pt]
\tikzstyle{every node}=[minimum width=4.5pt, inner sep=0pt, circle]
\draw (-1,0) node (v1) [draw, fill=gray, label=left:{\notag}] {};
\draw (0,0) node (v2) [draw, fill=gray, label=right:{\notag}] {};
\draw (0,1) node (v3) [draw, fill=gray, label=right:{\notag}] {};
\draw (-1,1) node (v4) [draw, fill=gray, label=left:{\notag}] {};
\draw (v1) -- (v2) -- (v3) -- (v4) -- (v1);
\draw (v1) -- (v3);
\draw (v2) -- (v4);
\end{tikzpicture} & \begin{tabular}{c}
$215$\\
\\
\end{tabular} &\begin{tabular}{c}
$323$\\
\\
\end{tabular} &
\begin{tikzpicture}[scale=0.5, line width=0.7pt]
\tikzstyle{every node}=[minimum width=4.5pt, inner sep=0pt, circle]
\draw (-1,0) node (v1) [draw, fill=gray, label=left:{\notag}] {};
\draw (-0.5,-1) node (v2) [draw, fill=gray, label=left:{\notag}] {};
\draw (0.5,-1) node (v3) [draw, fill=gray, label=right:{\notag}] {};
\draw (1,0) node (v4) [draw, fill=gray, label=right:{\notag}] {};
\draw (0,1) node (v5) [draw, fill=gray, label=left:{\notag}] {};
\draw (v1) -- (v2) -- (v3) -- (v4) -- (v5);
\end{tikzpicture} 
& 
\begin{tabular}{c}
$14$\\
\\
\end{tabular}
&
\begin{tabular}{c}
$14$\\
\\
\end{tabular} 
\\
\hline
\begin{tikzpicture}[scale=0.5, line width=0.7pt]
\tikzstyle{every node}=[minimum width=4.5pt, inner sep=0pt, circle]
\draw (0,0) node (v1) [draw, fill=gray, label=below:{\notag}] {};
\draw (-1,-0.5) node (v2) [draw, fill=gray, label=left:{\notag}] {};
\draw (1,-0.5) node (v3) [draw, fill=gray, label=right:{\notag}] {};
\draw (0,0.8) node (v4) [draw, fill=gray, label=left:{\notag}] {};
\draw (0,1.6) node (v5) [draw, fill=gray, label=left:{\notag}] {};
\draw (v1) -- (v2); 
\draw (v1) -- (v3);
\draw (v1) -- (v4);
\draw (v4) -- (v5);
\end{tikzpicture} 
&
\begin{tabular}{c}
$46$\\
\\
\end{tabular} & 
\begin{tabular}{c}
$62$\\
\\
\end{tabular}&
\begin{tikzpicture}[scale=0.5, line width=0.7pt]
\tikzstyle{every node}=[minimum width=4.5pt, inner sep=0pt, circle]
\draw (-1.25,0) node (v1) [draw, fill=gray, label=left:{\notag}] {};
\draw (-0.75,-1) node (v2) [draw, fill=gray, label=below:{\notag}] {};
\draw (0.75,-1) node (v3) [draw, fill=gray, label=below:{\notag}] {};
\draw (1.25,0) node (v4) [draw, fill=gray, label=right:{\notag}] {};
\draw (0,0.75) node (v5) [draw, fill=gray, label=above:{\notag}] {};
\draw (v1) -- (v2) -- (v3) -- (v4) -- (v5) -- (v1);
\end{tikzpicture} 
&
\begin{tabular}{c}
$126$\\
\\
\end{tabular}
&
\begin{tabular}{c}
$126$\\
\\
\end{tabular}&
\begin{tikzpicture}[scale=0.5, line width=0.7pt]
\tikzstyle{every node}=[minimum width=4.5pt, inner sep=0pt, circle]
\draw (0,0) node (v1) [draw, fill=gray, label=right:{\notag}] {};
\draw (-1,-0.75) node (v2) [draw, fill=gray, label=left:{\notag}] {};
\draw (1,-0.75) node (v3) [draw, fill=gray, label=right:{\notag}] {};
\draw (0,0.8) node (v4) [draw, fill=gray, label=left:{\notag}] {};
\draw (0,1.6) node (v5) [draw, fill=gray, label=left:{\notag}] {};
\draw (v1) -- (v2) -- (v3); 
\draw (v1) -- (v3);
\draw (v1) -- (v4);
\draw (v4) -- (v5);
\end{tikzpicture} & 
\begin{tabular}{c}
$102$\\
\\
\end{tabular}&
\begin{tabular}{c}
$162$\\
\\
\end{tabular}
\\
\hline
\begin{tikzpicture}[scale=0.5, line width=0.7pt]
\tikzstyle{every node}=[minimum width=4.5pt, inner sep=0pt, circle]
\draw (-1.25,0) node (v1) [draw, fill=gray, label=left:{\notag}] {};
\draw (-0.75,-1) node (v2) [draw, fill=gray, label=below:{\notag}] {};
\draw (0.75,-1) node (v3) [draw, fill=gray, label=below:{\notag}] {};
\draw (1.25,0) node (v4) [draw, fill=gray, label=right:{\notag}] {};
\draw (0,0.75) node (v5) [draw, fill=gray, label=above:{\notag}] {};
\draw (v1) -- (v2) -- (v3) -- (v4) -- (v5);
\draw (v1) -- (v4);
\end{tikzpicture} 
&\begin{tabular}{c}
$134$\\
\\
\end{tabular} 
& \begin{tabular}{c}
$245$\\
\\
\end{tabular}&
\begin{tikzpicture}[scale=0.5, line width=0.7pt]
\tikzstyle{every node}=[minimum width=4.5pt, inner sep=0pt, circle]
\draw (-1,0) node (v1) [draw, fill=gray, label=left:{\notag}] {};
\draw (-0.5,-1.2) node (v2) [draw, fill=gray, label=below:{\notag}] {};
\draw (0.5,-1.2) node (v3) [draw, fill=gray, label=below:{\notag}] {};
\draw (1,0) node (v4) [draw, fill=gray, label=right:{\notag}] {};
\draw (0,0.75) node (v5) [draw, fill=gray, label=left:{\notag}] {};
\draw (v1) -- (v2) -- (v3) -- (v4);
\draw (v5) -- (v2);
\draw (v5) -- (v3);
\end{tikzpicture} 
&\begin{tabular}{c}
$120$\\
\\
\end{tabular} & \begin{tabular}{c}
$300$\\
\\
\end{tabular}&
\begin{tikzpicture}[scale=0.5, line width=0.7pt]
\tikzstyle{every node}=[minimum width=4.5pt, inner sep=0pt, circle]
\draw (0,0) node (v1) [draw, fill=gray, label=above:{\notag}] {};
\draw (-1,1) node (v2) [draw, fill=gray, label=left:{\notag}] {};
\draw (-1,-1) node (v3) [draw, fill=gray, label=left:{\notag}] {};
\draw (1,1) node (v4) [draw, fill=gray, label=right:{\notag}] {};
\draw (1,-1) node (v5) [draw, fill=gray, label=right:{\notag}] {};
\draw (v1) -- (v2); 
\draw (v1) -- (v3);
\draw (v1) -- (v4);
\draw (v1) -- (v5);
\end{tikzpicture} 
& \begin{tabular}{c}
$263$\\
\\
\end{tabular}&
\begin{tabular}{c}
$371$\\
\\
\end{tabular}
\\
\hline
\end{tabular}
\end{center}
\caption{This table presents some small graphs together with their number of arithmetical structures ($\big| \mathcal{A} \big|$) and the size of the output of Algorithm~\ref{A1} ($\big|\min \mathcal{D}_{\geq 0}\big|$).}
    \label{tab:graphs1}
\end{table}
\vspace{-5mm}
Similarly for Tables \ref{tab:graphs2} and \ref{tab:graphs3}, therein we also present the total execution time for each graph and the average execution time for each element found by Algorithm \ref{A1} (amortized execution time).

\begin{table}[H]
\begin{center}
\begin{tabular}{|c|c|c|c|c|c|c|c|}
\hline
Graph & {\small $\big|\mathcal{A}\big|$} & {\small $\big| \min \mathcal{D}_{\geq 0}\big|$} & \begin{tabular}{c}
    {\small Total $\&$ amort.}\\
    {\small exec. Times}   
\end{tabular} & Graph & {\small $\big|\mathcal{A}\big|$} &{\small $\big| \min \mathcal{D}_{\geq 0}\big|$} & \begin{tabular}{c}
    {\small Total $\&$ amort.}\\
    {\small exec. Times}   
\end{tabular} \\
\hline

\begin{tikzpicture}[scale=0.5, line width=0.7pt]
\tikzstyle{every node}=[minimum width=4.5pt, inner sep=0pt, circle]
\draw (-1.25,0.2) node (v1) [draw, fill=gray, label=left:{\notag}] {};
\draw (-0.75,-1) node (v2) [draw, fill=gray, label=left:{\notag}] {};
\draw (0.75,-1) node (v3) [draw, fill=gray, label=right:{\notag}] {};
\draw (1.25,0.2) node (v4) [draw, fill=gray, label=right:{\notag}] {};
\draw (0,0.8) node (v5) [draw, fill=gray, label=above:{\notag}] {};
\draw (v2) -- (v3) -- (v5) -- (v2);
\draw (v5) -- (v1);
\draw (v5) -- (v4);
\end{tikzpicture}
& \begin{tabular}{c}
$257$\\
\\
\end{tabular}& \begin{tabular}{c}
$809$\\
\\
\end{tabular} & \begin{tabular}{c}
    29.7 sec.\\
    $\&$\\
    36.7 ms. 
\end{tabular} &
\begin{tikzpicture}[scale=0.5, line width=0.7pt]
\tikzstyle{every node}=[minimum width=4.5pt, inner sep=0pt, circle]
\draw (-1.3,0.2) node (v1) [draw, fill=gray, label=left:{\notag}] {};
\draw (-0.75,-0.8) node (v2) [draw, fill=gray, label=below:{\notag}] {};
\draw (0.75,-0.8) node (v3) [draw, fill=gray, label=below:{\notag}] {};
\draw (1.3,0.2) node (v4) [draw, fill=gray, label=right:{\notag}] {};
\draw (0,0.9) node (v5) [draw, fill=gray, label=above:{\notag}] {};
\draw (v1) -- (v2) -- (v3) -- (v4) -- (v5);
\draw (v3) -- (v1) -- (v4);
\end{tikzpicture}& \begin{tabular}{c}
$388$\\
\\
\end{tabular}&\begin{tabular}{c}
$845$\\
\\
\end{tabular}&
\begin{tabular}{c}
    27.5 sec.\\
    $\&$\\
    32.5 ms. 
\end{tabular}\\
\hline
\begin{tikzpicture}[scale=0.5, line width=0.7pt]
\tikzstyle{every node}=[minimum width=4.5pt, inner sep=0pt, circle]
\draw (-1.3,0) node (v1) [draw, fill=gray, label=left:{\notag}] {};
\draw (-0.75,-1) node (v2) [draw, fill=gray, label=below:{\notag}] {};
\draw (0.75,-1) node (v3) [draw, fill=gray, label=below:{\notag}] {};
\draw (1.3,0) node (v4) [draw, fill=gray, label=right:{\notag}] {};
\draw (0,0.8) node (v5) [draw, fill=gray, label=above:{\notag}] {};
\draw (v1) -- (v2) -- (v3) -- (v4) -- (v5) -- (v1);
\draw (v1) -- (v4);
\end{tikzpicture}
& \begin{tabular}{c}
$290$\\
\\
\end{tabular}&\begin{tabular}{c}
$864$\\
\\
\end{tabular} & \begin{tabular}{c}
    30.3 sec.\\
    $\&$\\
    35.0 ms. 
\end{tabular} &
\begin{tikzpicture}[scale=0.5, line width=0.7pt]
\tikzstyle{every node}=[minimum width=4.5pt, inner sep=0pt, circle]
\draw (-1.3,0) node (v1) [draw, fill=gray, label=left:{\notag}] {};
\draw (-0.75,-0.9) node (v2) [draw, fill=gray, label=below:{\notag}] {};
\draw (0.75,-0.9) node (v3) [draw, fill=gray, label=below:{\notag}] {};
\draw (1.3,0) node (v4) [draw, fill=gray, label=right:{\notag}] {};
\draw (0,0.9) node (v5) [draw, fill=gray, label=above:{\notag}] {};
\draw (v1) -- (v2) -- (v3) -- (v5) -- (v1);
\draw (v2) -- (v4) -- (v5);
\end{tikzpicture}& \begin{tabular}{c}
$571$\\
\\
\end{tabular} & \begin{tabular}{c}
$960$\\
\\
\end{tabular}&
\begin{tabular}{c}
    7.6 sec.\\
    $\&$\\
    7.9 ms. 
\end{tabular}
\\
\hline
\begin{tikzpicture}[scale=0.5, line width=0.7pt]
\tikzstyle{every node}=[minimum width=4.5pt, inner sep=0pt, circle]
\draw (0,0) node (v1) [draw, fill=gray, label=above:{\notag}] {};
\draw (-1,0.9) node (v2) [draw, fill=gray, label=above:{\notag}] {};
\draw (-1,-0.9) node (v3) [draw, fill=gray, label=below:{\notag}] {};
\draw (1,0.9) node (v4) [draw, fill=gray, label=above:{\notag}] {};
\draw (1,-0.9) node (v5) [draw, fill=gray, label=below:{\notag}] {};
\draw (v1) -- (v2) -- (v3) -- (v1); 
\draw (v1) -- (v4) -- (v5) -- (v1);
\end{tikzpicture}& \begin{tabular}{c}
    $835$\\
    \\
\end{tabular} & \begin{tabular}{c}
$1531$\\
\\
\end{tabular}&
\begin{tabular}{c}
    121.3 sec.\\
    $\&$\\
    79.2 ms. 
\end{tabular}
&
\begin{tikzpicture}[scale=0.5, line width=0.7pt]
\tikzstyle{every node}=[minimum width=4.5pt, inner sep=0pt, circle]
\draw (-1.3,0) node (v1) [draw, fill=gray, label=left:{\notag}] {};
\draw (-0.75,-0.9) node (v2) [draw, fill=gray, label=below:{\notag}] {};
\draw (0.75,-0.9) node (v3) [draw, fill=gray, label=below:{\notag}] {};
\draw (1.3,0) node (v4) [draw, fill=gray, label=right:{\notag}] {};
\draw (0,0.9) node (v5) [draw, fill=gray, label=above:{\notag}] {};
\draw (v5) -- (v1) -- (v2) -- (v3) -- (v4);
\draw (v3) -- (v1) -- (v4);
\end{tikzpicture}& \begin{tabular}{c}
$449$\\
\\
\end{tabular}&\begin{tabular}{c}
$1771$\\
\\
\end{tabular}&
\begin{tabular}{c}
    181.0 sec.\\
    $\&$\\
    102.2 ms. 
\end{tabular}
\\
\hline
\begin{tikzpicture}[scale=0.5, line width=0.7pt]
\tikzstyle{every node}=[minimum width=4.5pt, inner sep=0pt, circle]
\draw (-1.3,0) node (v1) [draw, fill=gray, label=left:{\notag}] {};
\draw (-0.75,-1) node (v2) [draw, fill=gray, label=below:{\notag}] {};
\draw (0.75,-1) node (v3) [draw, fill=gray, label=below:{\notag}] {};
\draw (1.3,0) node (v4) [draw, fill=gray, label=right:{\notag}] {};
\draw (0,0.8) node (v5) [draw, fill=gray, label=above:{\notag}] {};
\draw (v1) -- (v2) -- (v3) -- (v4) -- (v5) -- (v1);
\draw (v1) -- (v3);
\draw (v4) -- (v2);
\end{tikzpicture}& \begin{tabular}{c}
$987$\\
\\
\end{tabular}&\begin{tabular}{c}
$2524$\\
\\
\end{tabular} &
\begin{tabular}{c}
    115.3 sec.\\
    $\&$\\
    45.7 ms. 
\end{tabular}
&
\begin{tikzpicture}[scale=0.5, line width=0.7pt]
\tikzstyle{every node}=[minimum width=4.5pt, inner sep=0pt, circle]
\draw (-1.3,0) node (v1) [draw, fill=gray, label=left:{\notag}] {};
\draw (-0.75,-1) node (v2) [draw, fill=gray, label=below:{\notag}] {};
\draw (0.75,-1) node (v3) [draw, fill=gray, label=below:{\notag}] {};
\draw (1.3,0) node (v4) [draw, fill=gray, label=right:{\notag}] {};
\draw (0,0.8) node (v5) [draw, fill=gray, label=above:{\notag}] {};
\draw (v5) -- (v1) -- (v2) -- (v3) -- (v4);
\draw (v3) -- (v1) -- (v4) -- (v2);
\end{tikzpicture} & \begin{tabular}{c}
$1079$\\
\\
\end{tabular} &\begin{tabular}{c}
$3311$\\
\\
\end{tabular}&
\begin{tabular}{c}
    530.6 sec.\\
    $\&$\\
    160.2 ms. 
\end{tabular}
\\
\hline
\end{tabular}
\end{center}
\caption{In this table we also list the total and amortized execution times for these eight graphs on five vertices, presented in seconds and milliseconds respectively.}
    \label{tab:graphs2}
\end{table}

\begin{table}[H]
    
\begin{center}
\begin{tabular}{|c|c|c|c|c|c|c|c|}
\hline
Graph & {\small $\big|\mathcal{A}\big|$} & {\small $\big| \min \mathcal{D}_{\geq 0}\big|$} & \begin{tabular}{c}
    {\small Total $\&$ amort.}\\
    {\small exec. Times}   
\end{tabular} & Graph & {\small $\big|\mathcal{A}\big|$} &{\small $\big| \min \mathcal{D}_{\geq 0}\big|$} & \begin{tabular}{c}
    {\small Total $\&$ amort.}\\
    {\small exec. Times}   
\end{tabular} \\
\hline
\begin{tikzpicture}[scale=0.5, line width=0.7pt]
\tikzstyle{every node}=[minimum width=4.5pt, inner sep=0pt, circle]
\draw (-1.3,0) node (v1) [draw, fill=gray, label=left:{\notag}] {};
\draw (-0.75,-0.9) node (v2) [draw, fill=gray, label=below:{\notag}] {};
\draw (0.75,-0.9) node (v3) [draw, fill=gray, label=below:{\notag}] {};
\draw (1.3,0) node (v4) [draw, fill=gray, label=right:{\notag}] {};
\draw (0,0.9) node (v5) [draw, fill=gray, label=above:{\notag}] {};
\draw (v1) -- (v2) -- (v3) -- (v4) -- (v5) -- (v1);
\draw (v3) -- (v5) -- (v2);
\end{tikzpicture}&\begin{tabular}{c}
$489$\\
\\
\end{tabular}&\begin{tabular}{c}
$3647$\\
\\
\end{tabular}&\begin{tabular}{c}
51.6 min.\\
$\&$\\
0.84 sec.
\end{tabular}&
\begin{tikzpicture}[scale=0.5, line width=0.7pt]
\tikzstyle{every node}=[minimum width=4.5pt, inner sep=0pt, circle]
\draw (-1.3,0) node (v1) [draw, fill=gray, label=left:{\notag}] {};
\draw (-0.9,-1.1) node (v2) [draw, fill=gray, label=below:{\notag}] {};
\draw (0.8,-1.1) node (v3) [draw, fill=gray, label=below:{\notag}] {};
\draw (1.3,0) node (v4) [draw, fill=gray, label=right:{\notag}] {};
\draw (0,0.7) node (v5) [draw, fill=gray, label=above:{\notag}] {};
\draw (v1) -- (v2) -- (v3) -- (v4) -- (v2);
\draw (v3) -- (v1);
\draw (v2) -- (v5) -- (v3);
\end{tikzpicture} 
&\begin{tabular}{c}
$2181$\\
\\
\end{tabular}&\begin{tabular}{c}
$4942$\\
\\
\end{tabular}&\begin{tabular}{c}
32.3 min.\\
$\&$\\
0.39 sec.
\end{tabular}
\\
\hline
\begin{tikzpicture}[scale=0.5, line width=0.7pt]
\tikzstyle{every node}=[minimum width=4.5pt, inner sep=0pt, circle]
\draw (-1.3,0) node (v1) [draw, fill=gray, label=left:{\notag}] {};
\draw (-0.7,-0.9) node (v2) [draw, fill=gray, label=below:{\notag}] {};
\draw (0.7,-0.9) node (v3) [draw, fill=gray, label=below:{\notag}] {};
\draw (1.3,0) node (v4) [draw, fill=gray, label=right:{\notag}] {};
\draw (0,0.9) node (v5) [draw, fill=gray, label=above:{\notag}] {};
\draw (v5) -- (v1) -- (v2) -- (v3) -- (v4);
\draw (v1) -- (v4);
\draw (v5) -- (v4);
\draw (v2) -- (v5) --(v3);
\end{tikzpicture} & \begin{tabular}{c}
$1419$\\
\\
\end{tabular}&\begin{tabular}{c}
$7405$\\
\\
\end{tabular}&\begin{tabular}{c}
34.3 min.\\
$\&$\\
0.28 sec.
\end{tabular} &
\begin{tikzpicture}[scale=0.5, line width=0.7pt]
\tikzstyle{every node}=[minimum width=4.5pt, inner sep=0pt, circle]
\draw (-1.3,0) node (v1) [draw, fill=gray, label=left:{\notag}] {};
\draw (-0.75,-1) node (v2) [draw, fill=gray, label=below:{\notag}] {};
\draw (0.75,-1) node (v3) [draw, fill=gray, label=below:{\notag}] {};
\draw (1.3,0) node (v4) [draw, fill=gray, label=right:{\notag}] {};
\draw (0,0.8) node (v5) [draw, fill=gray, label=above:{\notag}] {};
\draw (v4) -- (v5) -- (v1) -- (v2) -- (v3) -- (v4);
\draw (v3) -- (v1) -- (v4) -- (v2);
\end{tikzpicture}&\begin{tabular}{c}
$1541$\\
\\
\end{tabular}&\begin{tabular}{c}
$8702$\\
\\
\end{tabular}&\begin{tabular}{c}
74.1 min.\\
$\&$\\
0.51 sec.
\end{tabular}\\
\hline
\begin{tikzpicture}[scale=0.5, line width=0.7pt]
\tikzstyle{every node}=[minimum width=4.5pt, inner sep=0pt, circle]
\draw (-1.3,0) node (v1) [draw, fill=gray, label=left:{\notag}] {};
\draw (-0.7,-1) node (v2) [draw, fill=gray, label=below:{\notag}] {};
\draw (0.7,-1) node (v3) [draw, fill=gray, label=below:{\notag}] {};
\draw (1.3,0) node (v4) [draw, fill=gray, label=right:{\notag}] {};
\draw (0,0.9) node (v5) [draw, fill=gray, label=above:{\notag}] {};
\draw (v4) -- (v5) -- (v1) -- (v2) -- (v3) -- (v4);
\draw (v3) -- (v1) -- (v4) -- (v2) -- (v5);
\end{tikzpicture} 
&\begin{tabular}{c}
$3325$\\
\\
\end{tabular} &\begin{tabular}{c}
$17349$\\
\\
\end{tabular} & \begin{tabular}{c}
228.8 min.\\
$\&$\\
0.79 sec.
\end{tabular}&
\begin{tikzpicture}[scale=0.5, line width=0.7pt]
\tikzstyle{every node}=[minimum width=4.5pt, inner sep=0pt, circle]
\draw (-1.3,0) node (v1) [draw, fill=gray, label=left:{\notag}] {};
\draw (-0.75,-1) node (v2) [draw, fill=gray, label=below:{\notag}] {};
\draw (0.75,-1) node (v3) [draw, fill=gray, label=below:{\notag}] {};
\draw (1.3,0) node (v4) [draw, fill=gray, label=right:{\notag}] {};
\draw (0,0.9) node (v5) [draw, fill=gray, label=above:{\notag}] {};
\draw (v4) -- (v5) -- (v1) -- (v2) -- (v3) -- (v4);
\draw (v5) -- (v3) -- (v1) -- (v4) -- (v2) -- (v5);
\end{tikzpicture} &\begin{tabular}{c}
$12231$\\
\\
\end{tabular} &\begin{tabular}{c}
$32701$\\
\\
\end{tabular} &\begin{tabular}{c}
797.6 min.\\
$\&$\\
1.46 sec.
\end{tabular}
\\
\hline
\end{tabular}
\end{center}
\caption{This table presents the last six connected graphs on five vertices. Note that the total execution time is presented in minutes and the amortized time is presented in seconds.}
\label{tab:graphs3}
\end{table}

The time complexity of Algorithm~\ref{A1} is difficult to approximate in general. For simple graphs, the time complexity is of the order of
\[
(n^{k}) \prod_{s\in [n]} \big|\min \mathcal{D}_{\geq 0}(L_s)\big| \max_{\bff{d}\in \min \mathcal{D}_{\geq 0}(L)}( \prod_{i\in [n]} d_i ),
\]
for some constant $k>0$. Note that for the complete graph with $n$ vertices, applying the twin vertices method describe above. Then the time complexity is of the order 
\[
(n^{k})\big|\min \mathcal{D}_{\geq 0}(A(K_{n-1}))\big| C^{2(2^n -1)}.
\] 
Improving the execution time for $K_n$ by a factor of order $\left|\min \mathcal{D}_{\geq 0}(A(K_{n-1}))\right|^{n-1}.$
This means that, for instance, instead of roughly $13.5$ hours of execution time for $K_5$, as seen in Table~\ref{tab:graphs3}. 
The same computer would take several months to finish computing Algorithm~\ref{A1} without using the symmetries of twin vertices. 

\noindent {\bf Acknowledgments}

The authors would like to thank the anonymous referee for their helpful comments.


\end{document}